\documentclass[12pt]{amsart}
\usepackage{SSdefn}
\usepackage{eucal}
\usepackage{tikz-cd}

\newcommand{\FI}{\mathbf{FI}}

\newcommand{\Res}{\mathrm{Res}}

\DeclareMathOperator{\maxdeg}{maxdeg}

\title{Regularity of $\FI$-modules and local cohomology}
\date{December 26, 2017}

\author{Rohit Nagpal}
\address{Department of Mathematics, University of Chicago, Chicago, IL}
\email{\href{mailto:nagpal@math.uchicago.edu}{nagpal@math.uchicago.edu}}
\urladdr{\url{http://math.uchicago.edu/~nagpal/}}

\author{Steven V Sam}
\address{Department of Mathematics, University of Wisconsin, Madison, WI}
\email{\href{mailto:svs@math.wisc.edu}{svs@math.wisc.edu}}
\urladdr{\url{http://math.wisc.edu/~svs/}}

\thanks{SS was partially supported by NSF grant DMS-1500069.}

\author{Andrew Snowden}
\address{Department of Mathematics, University of Michigan, Ann Arbor, MI}
\email{\href{mailto:asnowden@umich.edu}{asnowden@umich.edu}}
\urladdr{\url{http://www-personal.umich.edu/~asnowden/}}

\thanks{AS was partially supported by NSF grants DMS-1303082 and DMS-1453893 and a Sloan Fellowship.}

\subjclass[2010]{%
13D45, 
20C30
}

\begin{document}

\begin{abstract}
We resolve a conjecture of Ramos and Li that relates the regularity of an $\FI$-module to its local cohomology groups. This is an analogue of the familiar relationship between regularity and local cohomology in commutative algebra.
\end{abstract}

\maketitle

\section{Introduction}

Let $S$ be a standard-graded polynomial ring in finitely many variables over a field $\bk$, and let $M$ be a non-zero finitely generated graded $S$-module. It is a classical fact in commutative algebra that the following two quantities are equal (see \cite[\S 4B]{geomsyzygies}): 
\begin{itemize}[\hspace{8pt}$\bullet$]
\item The minimum integer $\alpha$ such that $\Tor^S_i(M, \bk)$ is supported in degrees $\le \alpha+i$ for all $i$.
\item The minimum integer $\beta$ such that $\rH^i_{\fm}(M)$ is supported in degrees $\le \beta-i$ for all $i$.
\end{itemize}
Here $\rH^i_{\fm}$ is local cohomology at the irrelevant ideal $\fm$. The quantity $\alpha=\beta$ is called the {\bf (Castelnuovo--Mumford) regularity} of $M$, and is one of the most important numerical invariants of $M$. In this paper, we establish the analog of the $\alpha=\beta$ identity for $\FI$-modules.

To state our result precisely, we must recall some definitions. Let $\FI$ be the category of finite sets and injections. Fix a commutative noetherian ring $\bk$. An {\bf $\FI$-module} over $\bk$ is a functor from $\FI$ to the category of $\bk$-modules. We write $\Mod_{\FI}$ for the category of $\FI$-modules. We refer to \cite{fimodules} for a general introduction to $\FI$-modules.

Let $M$ be an $\FI$-module. Define $\Tor_0(M)$ to be the $\FI$-module that assigns to $S$ the quotient of $M(S)$ by the sum of the images of the $M(T)$, as $T$ varies over all proper subsets of $S$. Then $\Tor_0$ is a right-exact functor, and so we can consider its left derived functors $\Tor_{\bullet}$. In \S \ref{sec:local-cohomology}, we explain how $\Tor_{\bullet}$ is the derived functor of a tensor product. We note that the $\FI$-module homology considered in \cite{castelnuovo-regularity} is the same as our $\Tor_{\bullet}$. We let $t_i(M)$ be the maximum degree occurring in $\Tor_i(M)$ (using the convention $t_i(M)=-\infty$ if $\Tor_i(M)=0$), and define the {\bf regularity} of $M$, denoted $\reg(M)$, to be the minimum integer $\rho$ such that $t_i(M) \le \rho+i$ for all $i$. We note that, while most $\FI$-modules have infinite projective (and $\Tor$) dimension, every finitely generated $\FI$-module  has finite regularity; see \cite[Theorem~A]{castelnuovo-regularity} or Corollary~\ref{cor:finite-regularity} below.

An element $x \in M(S)$ is {\bf torsion} if there exists an injection $f \colon S \to T$ such that $f_*(x)=0$. Let $\rH^0_{\fm}(M)$ be the maximal torsion submodule of $M$. Then $\rH^0_{\fm}$ is a left-exact functor, and so we can consider its right derived functors $\rH^{\bullet}_{\fm}$, which we refer to as {\bf local cohomology}. If $M$ is finitely generated then each $\rH^i_{\fm}(M)$ is finitely generated and torsion, and $\rH^i_{\fm}(M)=0$ for $i \gg 0$ (see Proposition~ \ref{prop:loccoh}). We let $h^i(M)$ be the maximum degree occurring in $\rH^i_{\fm}(M)$, with the convention that $h^i(M)=-\infty$ if $\rH^i_{\fm}(M)=0$.

We can now state the main result of this paper:

\begin{theorem} \label{thm:main-intro}
Let $M$ be a finitely generated $\FI$-module. Then
\addtocounter{equation}{-1}
\begin{subequations}
\begin{equation} \label{eq:main}
\reg(M) = \max \big( t_0(M), \max_{i \ge 0} (h^i(M)+i) \big).
\end{equation}
Moreover, we have
\begin{displaymath}
t_n(M) = n + \max_{i \ge 0} (h^i(M)+i)
\end{displaymath}
for all $n \gg 0$. In particular, 
\[
\max_{n > 0}(t_n(M) - n) = \max_{i \ge 0}(h^i(M) + i).
\]
\end{subequations}
\end{theorem}

\begin{remark}
If $M$ is a module over a polynomial ring in finitely many variables then one can omit the $t_0(M)$ on the right side of \eqref{eq:main}. However, it is necessary in the case of $\FI$-modules. Indeed, if $M$ is the $\FI$-module given by $M(S)=\bk$ for all $S$ and all injections act as the identity, then all local cohomology groups of $M$ vanish, so $h^i(M)=-\infty$ for all $i$, but $\reg(M)=t_0(M)=0$.
\end{remark}

\begin{remark}
Theorem~\ref{thm:main-intro} can be proved for $\FI$-modules presented in finite degrees. We have restricted ourselves to finitely generated modules to keep the paper less technical.
\end{remark}

\begin{remark}
The theorem was first conjectured by Li and Ramos \cite[Conjecture~1.3]{li-ramos}. In fact, they conjectured the result for $\FI_G$-modules, where $G$ is a finite group. The version for $\FI_G$-modules follows immediately from the version for $\FI$-modules, since local cohomology and regularity do not depend on the $G$-action: we clearly have $\Psi \rH^0_{\fm}(M) = \rH^0_{\fm}(\Psi M)$ and $\Psi \Tor_0(M) = \Tor_0(\Psi M)$ where $\Psi \colon \Mod_{\FI_G} \to \Mod_{\FI}$ is the forgetful functor, and since $\Psi$ preserves both the injective and the projective objects these results extend to higher derived functors as well.
\end{remark}

\subsection*{Overview of proof}

Using the structure theorem for $\FI$-modules (Theorem~\ref{thm:semires}), an easy spectral sequence argument shows that the regularity of $M$ is at most the maximum of $h^i(M)+i$. Theorem~\ref{thm:main-intro} essentially says that there is not too much cancellation in this spectral sequence.

In characteristic~0, one can see this as follows. Let $\bM_{\lambda}$ be the irreducible representation of $S_n$ corresponding to the partition $\lambda$. Let $\ell(\lambda)$ be the number of parts in $\lambda$. For a representation $V$ of $S_n$, define $\ell(V)$ to be the maximum $\ell(\lambda)$ over those $\lambda$ for which $\bM_{\lambda}$ occurs in $V$. Now consider the relevant spectral sequence. One can directly observe that various terms in the spectral sequence have different $\ell$ values, and so some representations must always survive on the subsequent page. This proves that there is not too much cancellation.

In positive characteristic, there does not seem to be a complete analog of $\ell$. However, we construct an invariant $\nu$ that has some of the same properties. This is one of the key insights of this paper. The invariant $\nu$ is strong enough to distinguish terms in the spectral sequence, and thus allows the characteristic~0 argument to be carried out.

\subsection*{Outline of paper} 

In \S \ref{sec:local-cohomology}, we review some basic results on local cohomology of $\FI$-modules. In \S \ref{sec:combinatorics}, we define the invariant $\nu$ mentioned above and establish some of its basic properties. These results are combined in \S \ref{sec:regularity} to obtain Theorem~\ref{thm:main-intro}.

\subsection*{Acknowledgments}

We thank Eric Ramos for pointing out an error in an earlier version of this paper, and we thank Peter Patzt for several helpful comments and a thorough reading of the first draft of this paper.

\section{Preliminaries on $\FI$-modules} \label{sec:local-cohomology}

We fix a commutative noetherian ring $\bk$ for the entirety of the paper. We set  $[0]=\emptyset$, and for each positive integer $n$, we set $[n]=\{1,\dots,n\}$. Let $\Rep(S_{\star})$ be the category of sequences of representations of the symmetric groups over $\bk$. Given $V_{\bullet}$ and $W_{\bullet}$ in $\Rep(S_{\star})$, we define their tensor product by
\begin{displaymath}
(V_{\bullet} \otimes W_{\bullet})_n = \bigoplus_{i+j=n} \Ind_{S_i \times S_j}^{S_n} (V_i \otimes W_j).
\end{displaymath}
Then $\otimes$ endows $\Rep(S_{\star})$ with a monoidal structure (this is easier to see using the equivalence described in \cite[(5.1.6), (5.1.8)]{expos}). Furthermore, there is a symmetry of this monoidal structure by switching the order of $V$ and $W$ and conjugating $S_i \times S_j$ to $S_j \times S_i$ via the element $\tau_{ij} \in S_n$ which swaps the order of the two subsets $1,\dots,i$ and $i+1,\dots,n$. We thus have notions of commutative algebra and module objects in $\Rep(S_{\star})$.

Let $\bA=\bk[t]$, where $t$ has degree~1. We regard $\bA$ as an object of $\Rep(S_{\star})$ by letting $S_n$ act trivially on $\bA_n=\bk$. In this way, $\bA$ is a commutative algebra object of $\Rep(S_{\star})$.  By an $\bA$-module, we will always mean a module object for $\bA$ in $\Rep(S_{\star})$. We write $\Mod_{\bA}$ for the category of $\bA$-modules. As shown in \cite[Proposition 7.2.5]{catgb}, the categories $\Mod_{\bA}$ and $\Mod_{\FI}$ are equivalent. We pass freely between the two points of view. We regard $\bk$ as an $\bA$-module in the obvious way ($t$ acts by~0). We denote by $\Tor_i(-)$ the $i$th left derived functor of $\bk \otimes_{\bA} -$ on the category of $\bA$-modules. One easily sees that this definition coincides with the one from the introduction.

There is essentially only one $\Tor$ computation that we will use, namely $\Tor_{\bullet}(\bk)$. Let $\sgn_n$ be the sign representation of $S_n$, which we regard as an object of $\Rep(S_{\star})$ supported in degree $n$.  There is an inclusion of $\bk$-modules $\sgn_1 \to \bA$. We can consider the resulting Koszul complex $\lw^{\bullet}(\sgn_1) \otimes \bA$ in the category $\Rep(S_{\star})$. We now describe this complex explicitly. As a $\bk$-module, $\bigotimes^n (\sgn_1)$ is freely spanned by the generators $i_1 \otimes \cdots \otimes i_n$, one for each permutation $(i_1, \ldots, i_n)$  of the first $n$ natural numbers. From this, one sees that $\lw^n(\sgn_1)=\sgn_n$. Thus the $\bk$-module \[(\lw^{k}(\sgn_1) \otimes \bA)_n = \Ind_{S_k \times S_{n-k}}^{S_n} (\sgn_k \otimes \bA_{n-k})\]  is freely spanned by $i_1 \wedge \cdots \wedge i_k$ where $\{i_1, \ldots, i_k\}$ is a subset of size $k$ of $[n]$. The differential is the usual alternating sum
\[
  i_1 \wedge \cdots \wedge i_k \mapsto \sum_{j=1}^k (-1)^j i_1 \wedge\cdots \widehat{\imath_j} \cdots \wedge i_k
\]
where $\widehat{\imath_j}$ means that we omit that term. So in degree $n$, the Koszul complex is the usual complex that calculates the reduced homology of the standard $(n-1)$-simplex. It is well known that the standard $(n-1)$-simplex has no nontrivial reduced homology for $n>0$. This implies that this complex is exact in degrees $>0$, and that its $0$th homology is just $\bk$. We thus have a resolution $\bA \otimes \sgn_{\bullet} \to \bk$. Since $\bA \otimes \sgn_p$ is acyclic with respect to the functor $\bk \otimes_{\bA} -$ (also see Theorem~\ref{thm:tor-vanishing}), we see that $\Tor_p(M)$ is equal to the Koszul homology of $M$. 

\begin{proposition} \label{prop:tor-koszul}
If $T$ is an $\bA/\bA_+$-module, then $\Tor_p(T) = T \otimes \sgn_p$.
\end{proposition}

\begin{proof}
By the paragraph above, we have
\[
\Tor_p(T) =  \rH_p(T \otimes_{\bA} (\bA \otimes \sgn_{\bullet}) ) = \rH_p(T \otimes \sgn_{\bullet}).
\]
Since  $\bA_{+} T = 0$, we see that all the differentials in $T \otimes \sgn_{\bullet}$ vanish. This shows that $\rH_p(T \otimes \sgn_{\bullet}) = T \otimes \sgn_p$, completing the proof.  
\end{proof}

The restriction functor from $\Mod_{\FI}$ to $\Rep(S_{\star})$ admits a left adjoint denoted  $\cI$. We call $\FI$-modules of the form $\cI(V)$ {\bf induced} $\FI$-modules. In terms of $\bA$-modules, we have $\cI(V)=\bA \otimes V$. For a representation $V$ of $S_d$ we have 
\begin{displaymath}
\cI(V)_n = \bk[\Hom_{\FI}([d],[n])]\otimes_{\bk[S_d]} V = V \otimes \bA_{n-d}.
\end{displaymath}
See \cite[Definition~2.2.2]{fimodules} for more details on $\cI(V)$; note that there the notation $M(V)$ is used in place of $\cI(V)$. We say that an $\FI$-module $M$ is {\bf semi-induced} if it has a finite length filtration where the quotients are induced. (Semi-induced modules have also been called $\sharp$-filtered modules in the literature.) In characteristic~$0$, induced modules are projective, and so semi-induced modules are induced.

In the introduction, we defined $\rH^0_{\fm}(M)$ to be the maximal torsion submodule of an $\FI$-module $M$. We now introduce $\Gamma_{\fm}$ as a synonym for $\rH^0_{\fm}$, as it is better suited to the derived functor notation $\rR \Gamma_{\fm}$. Note that $\rR^i \Gamma_{\fm}$ is exactly the same as $\rH^i_{\fm}$.

\begin{theorem} \label{thm:tor-vanishing}
Let $M$ be a finitely generated $\FI$-module. Then the following are equivalent: \begin{enumerate}[\indent \rm (a)]
\item $M$ is semi-induced.
\item $\rR \Gamma_{\fm}(M) = 0$.
\item $\Tor_i(M) = 0$ for $i >0$.
\end{enumerate}
\end{theorem}

\begin{proof}
The equivalence (a) $\Longleftrightarrow$ (b) is proven in \cite[Proposition~5.12]{li-ramos}, and the equivalence (a) $\Longleftrightarrow$ (c) is established in \cite[Theorem~B]{ramos} (and independently in \cite[Theorem~1.3]{liyu}).
\end{proof}

\begin{lemma} \label{lem:fg-homology1}
Suppose $\cA$ is a locally noetherian abelian category. Let $M$ be a bounded chain complex in $\cA$ with finitely generated homologies. Then there exists a bounded complex $N$ with finitely generated terms which is quasi-isomorphic to $M$. Moreover, we may also assume that $N_i$ is a subobject of $M_i$ for each $i$.
\end{lemma}

\begin{proof} We proceed by induction on the length of the complex. The base case when $M$ is supported in at most one homological degree is trivial. Now suppose the length of $M$ is at least 2. By shifting, we may assume that the smallest $i$ such that $M_{i}$ is non-zero is 0. Let $N_{0}$ be a finitely generated subobject of $M_{0}$ such that $N_{0}/(d_{1}(M_1) \cap N_0) \cong \rH_0(M)$. Let $M'$ be the complex supported in positive degrees such that $M'_{i} = M_i$ if $i >1$, $M'_{1} = d_1^{-1}(d_1(M_1) \cap N_0)$, and $M'_0 = 0$. By induction on length, there exists a complex $N'$ with finitely generated terms which is quasi-isomorphic to $M'$ such that $N'_{i}$ is a subobject of $M'_i$ for each $i$. Since the map $d_1 \colon N_1' \to M_0$ factors through $\rH_1(N') = \rH_1(M') = d_1^{-1}(d_1(M_1) \cap N_0)/(d_2(M_2))$, we see that $d_1(N'_1) = N_0 \cap d_1(M_1)$. Set $N_i = N_i'$ for $i>1$. Then $N$ has all of the required properties. 
\end{proof}

\begin{lemma} \label{lem:fg-homology}
Let $M$ be a bounded complex of $\FI$-modules. Suppose all cohomology groups are finitely generated torsion $\FI$-modules. Then $M$ is quasi-isomorphic to a bounded complex of finitely generated torsion $\FI$-modules.
\end{lemma}

\begin{proof}
For an $\FI$-module $N$, let $N^{\le n}$ be the natural $\FI$-module defined by
\begin{displaymath}
(N^{\le n})_k = \begin{cases} N_k & \text{if $k \le n$} \\
0 & \text{if $k>n$} \end{cases}.
\end{displaymath}
It is clear that the functor $(-)^{\le n}$ is exact. We note that there is a natural surjection $N \to N^{\le n}$.

Over a noetherian ring, the category of $\FI$-modules is locally noetherian \cite[Theorem~A]{fi-noeth}. Thus by the previous lemma, we may assume that the terms of $M$ are finitely generated. Let $n$ be large enough so that that all cohomology groups of $M$ are supported in degrees $\le n$. Then $M \to M^{\le n}$ is a quasi-isomorphism, and $M^{\le n}$ is a bounded complex of torsion modules. Finally, apply Lemma~\ref{lem:fg-homology1} to $M^{\le n}$ to conclude that $M$ is quasi-isomorphic to a bounded complex of finitely generated torsion modules.
\end{proof}

\begin{theorem}[Structure theorem for $\FI$-modules] \label{thm:semires}
Let $M$ be a finitely generated $\FI$-module over a noetherian ring $\bk$. Then, in the derived category of $\FI$-modules, there is an exact triangle $T \to M \to F \to$ such that 
\begin{enumerate}[\indent \rm (a)]
\item $T$ is a bounded complex of finitely generated torsion modules supported in nonnegative degrees.
\item $F$ is a bounded complex of finitely generated semi-induced modules supported in nonnegative degrees.
\end{enumerate}
\end{theorem}

In characteristic~$0$, this theorem was proved in \cite{symc1}.

\begin{proof}
  In \cite[Theorem~A, part (B)]{nagpal},  a complex $F$ and a map $M \to F^0$ is constructed such that $F$ satisfies the condition in (b) and the augmented complex $M \to F^0 \to  F^1 \to \cdots$ is exact in high enough degrees (in the reference the terminology ``$\sharp$-filtered'' is used for semi-induced modules).  Since $M$ is supported in cohomological degree 0, we see that the augmented complex $M \to F^0 \to  F^1 \to \cdots$ is the mapping cone of the natural map  $M \to F$ of complexes. Since $\mathrm{Cone}(M \to F)$ is exact in high enough degrees, there exists a quasi-isomorphic complex $T$ satisfying the condition in (a) by Lemma~\ref{lem:fg-homology}. The theorem now follows as $\mathrm{Cone}(M \to F) \to M \to F \to$ is an exact triangle.
\end{proof}

The following corollary is a weaker version of \cite[Theorem~A]{castelnuovo-regularity}.

\begin{corollary}
\label{cor:finite-regularity}
A finitely generated $\FI$-module has finite regularity.
\end{corollary}

\begin{proof}
Using Theorem~\ref{thm:semires} and a d\'evissage argument, one is reduced to the case of induced $\FI$-modules, which obviously have finite regularity, and $\bA/\bA_+$-modules, which have finite regularity by Proposition~\ref{prop:tor-koszul}.
\end{proof}

\begin{proposition} \label{prop:loccoh}
Let $M$ be a finitely generated $\FI$-module, and let $T \to M \to F \to$ be the triangle in Theorem~\ref{thm:semires}. Then $\rH^i_{\fm}(M)=\rH^i(T)$. In particular, $\rH^i_{\fm}(M)$  is finitely generated for all $i$ and vanishes for $i \gg 0$.
\end{proposition}

\begin{proof}
This follows from Theorem~\ref{thm:tor-vanishing} and the fact that $\rR \Gamma_{\fm}(N)=N$ if $N$ is a torsion $\FI$-module. See also \cite[Theorem~E]{li-ramos}.
\end{proof}

For a non-zero graded $\bk$-module $M$, we let $\maxdeg(M)$ be the maximum degree in which $M$ is non-zero, or $\infty$ if $M$ is non-zero in arbitrarily high degrees. We also put $\maxdeg(M)=-\infty$ if $M = 0$. With this notation, we have
\begin{displaymath}
\reg(M)=\max_{i \ge 0} \left[ \maxdeg(\Tor_i(M))-i \right], \qquad
h^i(M)=\maxdeg \rH^i_{\fm}(M).
\end{displaymath}

\section{A result on symmetric group representations} \label{sec:combinatorics}

Over a field of characteristic~$0$, representations of symmetric groups decompose as a direct sum of simple representations, and the simples are indexed by partitions. Often, the number of rows in the partitions that appear gives useful information about the representation. Our goal is to extend the notion of ``the number of rows'' to a more general ring. Call a two-sided ideal $I \subseteq \bk[S_p]$ {\bf good} if the following properties hold: 
\begin{enumerate}
\item $I$ is idempotent,
\item $I$ annihilates $\sgn_p$,
\item $I$ does not annihilate $\Ind_{S_{p-1}}^{S_p}(\sgn_{p-1}) \otimes_{\bk} M$ for any non-zero $\bk$-module $M$,
\item $I$ is $\bk$-flat (and thus $\bk$-projective).
\end{enumerate}
We show that if $\bk[S_p]$ has a good ideal then for a $\bk[S_n]$-module $M$ and $n-\frac{n}{p} \le k \le n$, we can make sense of the number of rows in $M$ being equal to $k$.

\begin{proposition}
If $2$ is invertible in $\bk$, then there is a good ideal in $\bk[S_2]$.
\end{proposition}

\begin{proof}
Let $\rN=1+(1,2)$ be the norm element of $\bk[S_2]$, and let $I$ be the two-sided ideal generated by $\rN$. We verify that $I$ is good:
\begin{enumerate}
\item We have $\rN^2=2 \rN$, and so, since~2 is invertible, $I$ is idempotent.
\item It is clear that $\rN$ annihilates $\sgn_2$, and so $I$ does as well.
\item If $M$ is a $\bk$-module then $\Ind_{S_1}^{S_2}(\sgn_1) \otimes_\bk M = \bk[S_2] \otimes_\bk M$, which is clearly not annihilated by $\rN$.
\item As a $\bk$-module, $I$ is free of rank~1, and thus $\bk$-flat. \qedhere.
\end{enumerate}
\end{proof}

\begin{proposition}
If $3$ is invertible in $\bk$, then there is a good ideal in $\bk[S_3]$.
\end{proposition}

\begin{proof}
Let $\rN=\sum_{\sigma \in S_3} \sigma$ be the norm element of $\bk[S_3]$. Note that it is central. Let $I$ be the two-sided ideal generated by 
\begin{displaymath}
\tau = (1 + (1,2))(1 + (1,3)) - \tfrac{2}{3} \rN.
\end{displaymath}
Note that $\tfrac{2}{3}$ makes sense as we have assumed~$3$ to be invertible in $\bk$. We now verify that $I$ is good:
\begin{enumerate}
\item A straightforward computation shows that $\tau^2=\tau$, and so $I$ is idempotent.
\item Both $(1+(1,2))$ and $\rN$ annihilate $\sgn_3$, so the same is true for $I$. \item We have $\Ind_{S_2}^{S_3}(\sgn_2) \otimes_\bk M \cong M^{\oplus 3}$ where 
\[
\sigma \cdot (m_1,m_2,m_3) = \sgn(\sigma) (m_{\sigma^{-1}(1)}, m_{\sigma^{-1}(2)}, m_{\sigma^{-1}(3)}).
\]
Let $x \in M$ be any non-zero element. Then $\tau \cdot (x,0,0) = (x,-x,0) \ne 0$, so $I$ does not annihilate $\Ind_{S_2}^{S_3}(\sgn_2) \otimes_\bk M$.
\item We first claim that $I$ is equal to the ideal $J$ generated by the differences of two transpositions. The sum of the coefficients of the odd (or even) permutations appearing in $3 \tau$ is zero. This shows that $I \subset J$. The reverse inclusion $J \subset I$ follows from the following identity \[(1,3) - (1,2) = (1,3) \tau - \tau (1,2). \]  This establishes the claim. Clearly, we have  $\bk[S_3]/J \cong \bk^2$. This implies that $\bk[S_3]/I \cong \bk^2$. Thus, as a $\bk$-module, $I$ is a summand of $\bk[S_3]$, and therefore $\bk$-flat. \qedhere
\end{enumerate}
\end{proof}

Throughout the rest of this section, we fix an integer $p \ge 1$ and a good ideal $I$ of $\bk[S_p]$. If $pr \le n$, we define $I_n(r)$ to be the two-sided ideal of $\bk[S_n]$ generated by $I^{\boxtimes r}$ under the inclusion $\bk[S_p^{\times r}] \hookrightarrow \bk[S_n]$. For convenience, we set $I_n(r) = 0$ if $pr > n$. It is clear that $I_n(r)$ is idempotent.

\begin{definition}
\label{def:nu}
Let $M$ be a $\bk[S_n]$-module. We define $\nu(M)=n-r$ if $M$ is not annihilated by $I_n(r)$ but is annihilated by $I_n(s)$ for all $r<s$, and we set $\nu(0) = \infty$.
\end{definition}

\begin{proposition} \label{prop:nuses}
Consider an exact sequence
\begin{displaymath}
0 \to M_1 \to M_2 \to M_3 \to 0
\end{displaymath}
of $\bk[S_n]$-modules. Then $M_2$ is annihilated by $I_n(r)$ if and only if both $M_1$ and $M_3$ are. Consequently,
\begin{displaymath}
\nu(M_2) = \min(\nu(M_1), \nu(M_3)).
\end{displaymath}
\end{proposition}

\begin{proof}
If $M_2$ is annihilated by $I_n(r)$ then obviously $M_1$ and $M_3$ are. Suppose that $M_1$ and $M_3$ are annihilated by $I_n(r)$. Then the image of $I_n(r) M_2$ in $M_3$ vanishes, and so $I_n(r) M_2 \subset M_1$, and so $I_n(r)^2 M_2=0$. But $I_n(r)^2=I_n(r)$, and so $M_2$ is annihilated by $I_n(r)$.
\end{proof}

\begin{lemma} \label{lem:rth-iterate}
Let $N$ be any non-zero $\bk$-module. Then the ideal $I^{\boxtimes r}$ of $\bk[S_p^{\times r}]$ does not annihilate $(\Ind_{S_{p-1}}^{S_p} \sgn_{p-1})^{\boxtimes r} \otimes_{\bk}  N$.
\end{lemma}

\begin{proof} 
This follows by induction on $r$ and the definition of good.
\end{proof}

The following proposition is motivated by \cite[Proposition~3.1]{castelnuovo-regularity}.

\begin{proposition}
\label{prop:indnu}
Let $M$ be a non-zero representation of $S_d$, let $(p-1)d \le k$, put $n=k+d$. Then we have $\nu(M \otimes \sgn_k)=k$. 
\end{proposition}

\begin{proof}
Set $\nu = \nu(M \otimes \sgn_k)$. We first show that $I_n(r)$ does not annihilate $M \otimes \sgn_k$ for $r \le d$. The Mackey decomposition theorem gives
\[
\Res_{S_p^{\times r}}^{S_n}(M \otimes \sgn_k) = \bigoplus_{g \in (S_d \times S_k) \backslash S_n / S_p^{\times r}} \Ind^{S_p^{\times r}}_{S_p^{\times r} \cap (S_d \times S_k)^g} \Res^{(S_d \times S_k)^g}_{S_p^{\times r} \cap (S_d \times S_k)^g} (M \otimes \sgn_k)^g
\]
where the sum is over double coset representatives, and $(-)^g$ means conjugation by $g$. Taking $g$ so that $S_p^{\times r} \cap (S_d \times S_k)^g = S_{p-1}^{\times r}$, we see that it contains $(\Ind_{S_{p-1}}^{S_p} \sgn_{p-1})^{\boxtimes r} \otimes_{\bk}  M$ as a direct summand. Since $I_n(r)$ is generated by $I^{\boxtimes r}$, it suffices to show that $I^{\boxtimes r}$ does not annihilate this direct summand. But this follows from Lemma~\ref{lem:rth-iterate}.

Now we show that $I_n(r)$ annihilates $M \otimes \sgn_k$ if $n/p \ge r > d$.  Note that $\Res_{S_p^{\times r}}^{S_n}(M \otimes \sgn_k)$ decomposes naturally into a finite direct sum of $\bk[S_p^{\times r}]$-modules of the form $\boxtimes_{i=1}^r M_i$. Since $r >d$, at least one $M_i$ is isomorphic to $\sgn_p$ for each such direct summand. Thus $I^{\boxtimes r}$ annihilates each such direct summand. This shows that $I_n(r)$ annihilates $M \otimes \sgn_k$, completing the proof.
\end{proof}

\begin{remark}
Our invariant $\nu$ is an attempt to extend the notion of ``minimum number of rows in a simple object'' away from characteristic zero. To see this, let notation be as in  Proposition~\ref{prop:indnu}. In characteristic~$0$, the partitions in $M$ have $d$ boxes. Thus, by the Pieri rule, every partition appearing in $M \otimes \sgn_k$ has at least $k$ rows, and some have exactly $k$ rows. 

Since $I_n(r) = 0$ for $r> n/p$, our invariant $\nu$ can't distinguish between partitions with at most $n - \frac{n}{p}$ rows. 
\end{remark}

\section{The main theorem} \label{sec:regularity}

The aim of this section is to prove Theorem~\ref{thm:main-intro}. Before beginning we note that if $M$ is a graded $\bk$-module and $M[\frac{1}{2}]$ and $M[\frac{1}{3}]$ are the localizations of $M$ obtained by inverting 2 and~3, respectively, then 
\[
\maxdeg(M) = \max(\maxdeg(M[\tfrac{1}{2}]), \maxdeg(M[\tfrac{1}{3}])).
\]
(Proof: the kernel of the localization map $M \to M[\frac{1}{p}]$ is the set of elements annihilated by a power of $p$; if $x \in M$ is annihilated by both $2^n$ and $3^m$ then $x=0$, since $2^n$ and $3^m$ are coprime.) Localization commutes with Tor and local cohomology, so it suffices to prove Theorem~\ref{thm:main-intro} assuming that either 2 or 3 is invertible in $\bk$. In particular, in the remainder of this section, we may assume that $\bk[S_p]$ has a good ideal for either $p=2$ or $p=3$.

For a complex $M$ of $\FI$-modules, we define
\begin{displaymath}
\Tor_n(M)=\rH^{-n}(M \otimes^{\rL}_{\bA} \bk).
\end{displaymath}
(We use cohomological indexing throughout this section.) The {\bf regularity} of $M$ is the minimal $\rho$ so that $\maxdeg(\Tor_n(M)) \le n+\rho$ for all $n \in \bZ$.

\begin{lemma} \label{lem:tornu}
Let $M$ be a finite length complex of finitely generated torsion $\FI$-modules. Let $m$ be minimal such that $M^m \ne 0$. Then $\nu(\Tor_n(M)) \ge n+m$ for all $n \gg 0$.
\end{lemma}

\begin{proof}
Using the Koszul complex, we see that $\Tor_n(M)$ is a subquotient of
\begin{displaymath}
\bigoplus_{j \ge 0} M^{j-n} \otimes \sgn_j.
\end{displaymath}
Only the terms with $j \ge n+m$ contribute. Each of these has $\nu \ge n+m$ by Proposition~\ref{prop:indnu}, and this passes to subquotients by Proposition~\ref{prop:nuses}.
\end{proof}

The following lemma recovers \cite[Theorem~2]{ganli-homology}.

\begin{lemma} \label{lem:torsion-reg}
Let $M$ be a finitely generated non-zero torsion $\FI$-module, and let $\rho=\maxdeg(M)$. Then the regularity of $M$ is $\rho$, and for $n \gg 0$ we have
\begin{displaymath}
\nu(\Tor_n(M)_{n+\rho}) = n.
\end{displaymath}
\end{lemma}

\begin{proof}
Let $M_1$ be the degree $\rho$ piece of $M=M_2$, and let $M_3=M_2/M_1$. If $M_3 =  0$, then we are done by Proposition~\ref{prop:tor-koszul}. Now assume $M_3 \neq 0$. By induction on $\rho$, we can assume $\reg(M_3)<\rho$. We have an exact sequence
\begin{displaymath}
\Tor_{n+1}(M_3)_{n+\rho} \to \Tor_n(M_1)_{n+\rho} \to \Tor_n(M_2)_{n+\rho} \to 0.
\end{displaymath}
Note that $\Tor_n(M_3)_{n+\rho}=0$ by the bound on the regularity of $M_3$, which is why we have a 0 on the right above. Since $M_1$ is concentrated in one degree, we have $\Tor_n(M_1)=M_1 \otimes \sgn_n$. So by Proposition~\ref{prop:indnu}, the lemma is true for $M_1$. By Lemma~\ref{lem:tornu}, the leftmost term above has $\nu \ge n+1$. Since the middle term has $\nu=n$, we see (from Proposition~\ref{prop:nuses}) that the rightmost term is non-zero and has $\nu=n$, which completes the proof.
\end{proof}

\begin{proposition} \label{prop:torsion-complex}
Let $M$ be a finite length complex of finitely generated torsion $\FI$-modules. Put
\begin{displaymath}
\rho=\max_{i \in \bZ}(i+\maxdeg{\rH^i(M)}).
\end{displaymath}
Then the regularity of $M$ is $\rho$. Moreover, if $r$ is minimal so that $\rho=r+\maxdeg{\rH^r(M)}$ then
\begin{displaymath}
\nu(\Tor_n(M)_{n+\rho})=n+r
\end{displaymath}
for all $n \gg 0$.
\end{proposition}

\begin{proof}
Let $j$ be the minimal index so that $\rH^j(M) \ne 0$; we may as well assume that $M^i=0$ for $i<j$. Let $M_1$ be the kernel of $d \colon M^j \to M^{j+1}$, regarded as a complex concentrated in degree $j$, let $M_2=M$, and let $M_3=M_2/M_1$, so that we have a short exact sequence of complexes. Note that $\rH^j(M_1) \to \rH^j(M_2)$ is an isomorphism and $\rH^i(M_2) \to \rH^i(M_3)$ is an isomorphism for all $i>j$. Since $M_3$ has fewer non-zero cohomology groups than $M_2$, we can assume (by induction) that the proposition holds for $M_3$. The proposition holds for $M_1$ by Lemma~\ref{lem:torsion-reg}. We have an exact sequence
\begin{displaymath}
\Tor_{n+1}(M_3)_{n+\rho} \to \Tor_n(M_1)_{n+\rho} \to \Tor_n(M_2)_{n+\rho} \to \Tor_n(M_3)_{n+\rho} \to 0.
\end{displaymath}
Note that $\Tor_{n-1}(M_1)_{n+\rho}=0$, since the regularity of $M_1$ is at most $\rho$, which is why we have a 0 on the right. We now consider two cases:
\begin{itemize}
\item {\it Case 1:} $j=r$. We then have that $\nu(\Tor_n(M_1)_{n+\rho})=n+r$. By Lemma~\ref{lem:tornu}, $\nu(\Tor_{n+1}(M_3)_{n+\rho})>n+r$. If there exists $s>r$ such that $\rho=s+\maxdeg{\rH^s(M)}$  (and is chosen to be the smallest such $s$) then $M_3$ has regularity $\rho$ and $\nu(\Tor_n(M_3)_{n+\rho}) = n+s>n+r$ for $n \gg 0$; otherwise, $M_3$ has regularity $<\rho$ and $\Tor_n(M_3)_{n+\rho}=0$. Thus the two outside terms in the above 4-term sequence have $\nu>n+r$ (or vanish), and so $\Tor_2(M_2)_{n+\rho}$ is non-zero and has $\nu=n+r$.
\item {\it Case 2:} $j \ne r$. In this case, $M_1$ has regularity $<\rho$, and so $\Tor_n(M_1)_{n+\rho}=0$. Thus $\Tor_n(M_2)_{n+\rho}=\Tor_n(M_3)_{n+\rho}$, and the result follows by the inductive hypothesis. \qedhere
\end{itemize}
\end{proof}

We now prove our main result.

\begin{proof}[Proof of Theorem~\ref{thm:main-intro}] 
Let $T \to M \to F \to $ be the exact triangle as in Theorem~\ref{thm:semires}. By taking $\Tor$ we get a long exact sequence
\[
\cdots \to \Tor_n(T) \to \Tor_n(M) \to \Tor_n(F) \to \cdots.
\]
Note that $F$ is represented by a bounded complex of semi-induced modules and higher $\Tor$ groups of semi-induced modules are zero. Hence $F \otimes^{\rL}_{\bA}\bk$ is computed by the usual tensor product $F \otimes_{\bA}\bk$. Since $F$ is concentrated in non-negative cohomological degrees, this shows that $ \Tor_n(F) =  0$ for $n >0$. Thus, by the long exact sequence above, we have $\Tor_n(T) = \Tor_n(M)$ for $n>0$. Thus
\begin{displaymath}
\reg(M)=\max(t_0(M), \reg(T)).
\end{displaymath}
By Proposition~\ref{prop:loccoh}, we have $\rH^i(T) = \rH^i_{\fm}(M)$ for all $i$, and so $\maxdeg(\rH^i(T))=h^i(M)$. The theorem therefore follows from Proposition~\ref{prop:torsion-complex}.
\end{proof}


\begin{thebibliography}{CEFN}

\bibitem[CE]{castelnuovo-regularity} Thomas Church, Jordan S.~Ellenberg, Homology of FI-modules, {\em Geom. Topol.} {\bf 21} (2017), 2373--2418, \arxiv{1506.01022v2}.

\bibitem[CEF]{fimodules} Thomas Church, Jordan S.~Ellenberg, Benson Farb, FI-modules and stability for representations of symmetric groups, {\it Duke Math. J.} {\bf 164} (2015), no. 9, 1833--1910, \arxiv{1204.4533v4}.

\bibitem[CEFN]{fi-noeth} Thomas Church, Jordan S. Ellenberg, Benson Farb, and Rohit Nagpal, FI-modules over Noetherian rings, {\it Geom. Topol.} {\bf 18} (2014) 2951--2984, \arxiv{1210.1854v2}.

\bibitem[Ei]{geomsyzygies} David Eisenbud, \emph{The Geometry of Syzygies}, Graduate Texts in Mathematics {\bf 229}, Springer--Verlag, 2005.

\bibitem[GL]{ganli-homology} Wee Liang Gan, Liping Li, A remark on $\FI$-module homology, \emph{Michigan Math. J.} {\bf 65} (2016), no. 4, 855--861, \arxiv{1505.01777v4}.

\bibitem[LR]{li-ramos} Liping Li, Eric Ramos,  Depth and the local cohomology of $\FI_G$-modules, \arxiv{1602.04405v3}.

\bibitem[LY]{liyu} Liping Li, Nina Yu, Filtrations and homological degrees of FI-modules, \emph{J. Algebra} {\bf 472} (2017), 369--398, \arxiv{1511.02977v3}.

\bibitem[N]{nagpal} Rohit Nagpal, FI-modules and the cohomology of modular $S_n$-representations, \arxiv{1505.04294v1}.

\bibitem[R]{ramos} Eric Ramos, Homological invariants of $\FI$-modules and $\FI_G$-modules, \arxiv{1511.03964v3}

\bibitem[SS1]{symc1}
Steven V Sam, Andrew Snowden, GL-equivariant modules over polynomial rings in infinitely many variables, {\it Trans. Amer. Math. Soc.} {\bf 368} (2016), 1097--1158, \arxiv{1206.2233v3}.

\bibitem[SS2]{expos}
Steven V Sam, Andrew Snowden, Introduction to twisted commutative algebras, \arxiv{1209.5122v1}.

\bibitem[SS3]{catgb} Steven V Sam, Andrew Snowden, Gr\"obner methods for representations of combinatorial categories, {\it J. Amer. Math. Soc.} {\bf 30} (2017), 159--203, \arxiv{1409.1670v3}.

\end{thebibliography}
\end{document}